\documentclass[12pt,twoside]{amsart}
\usepackage{amssymb}
\usepackage{amscd}
\usepackage{xypic}

\title{Non-vanishing theorem for log canonical pairs}
\author{Osamu Fujino} 
\date{2009/12/1}
\address{Department of Mathematics, Faculty of 
Science, Kyoto University, 
Kyoto 606-8502 Japan} 
\email{fujino@math.kyoto-u.ac.jp}
\newcommand{\Supp}[0]{{\operatorname{Supp}}}
\newcommand{\Bs}[0]{{\operatorname{Bs}}}
\newcommand{\Exc}[0]{{\operatorname{Exc}}}
\newcommand{\mult}[0]{{\operatorname{mult}}}
\newtheorem{thm}{Theorem}[section]
\newtheorem{lem}[thm]{Lemma}

\newtheorem{cla}{Claim}

\theoremstyle{definition}

\newtheorem{rem}[thm]{Remark}
\newtheorem*{ack}{Acknowledgments}      
\newtheorem*{notation}{Notation}

\begin{document}
\bibliographystyle{amsalpha+}

\maketitle

\begin{abstract}
We obtain a correct generalization of 
Shokurov's non-vanishing theorem for log canonical 
pairs. 
It implies the base point free 
theorem for log canonical pairs. 
We also prove the rationality theorem for log canonical pairs. 
As a corollary, we obtain 
the cone theorem for 
log canonical pairs. 
We do not need Ambro's theory of quasi-log varieties. 
\end{abstract} 

\tableofcontents
\section{Introduction}

The following theorem is the main theorem of this paper. 
It is a generalization of Shokurov's non-vanishing theorem 
for log canonical pairs. 
This new non-vanishing theorem greatly simplifies  
the proof of the fundamental theorems for log canonical 
pairs. In this paper, we do not need Ambro's framework 
of {\em{quasi-log varieties}} in \cite{ambro}. 

We will work over $\mathbb C$, the complex 
number field, throughout this paper. 

\begin{thm}[Non-vanishing theorem]\label{31}
Let $X$ be a normal projective variety and 
$B$ an effective $\mathbb Q$-divisor 
on $X$ such that 
$(X, B)$ is log canonical. 
Let $L$ be a nef Cartier 
divisor on $X$. 
Assume that $aL-(K_X+B)$ is ample for some 
$a>0$. 
Then the base locus of the linear system 
$|mL|$ contains no lc centers of $(X, B)$ for every $m\gg 0$, 
that is, 
there is a positive integer $m_0$ such that 
$|mL|$ contains no lc centers of $(X, B)$ for every $m\geq m_0$. 
\end{thm}

By this new non-vanishing theorem, we can easily obtain the base point 
free theorem for log canonical pairs. 

\begin{thm}[Base point free theorem]\label{main1}
Let $X$ be a normal projective variety and 
$B$ an effective $\mathbb Q$-divisor 
on $X$ such that 
$(X, B)$ is log canonical. 
Let $L$ be a nef Cartier 
divisor on $X$. 
Assume that $aL-(K_X+B)$ is ample for some 
$a>0$. 
Then the linear system 
$|mL|$ is base point free for every $m\gg 0$. 
\end{thm}

Theorem \ref{main1} is a special case of \cite[Theorem 5.1]{ambro}. 
We can also prove the rationality theorem for log canonical 
pairs without any difficulties. It is a special 
case of \cite[Theorem 5.9]{ambro}. 

\begin{thm}[Rationality theorem]\label{61}
Let $(X, B)$ be a projective log canonical pair such that 
$a(K_X+B)$ is Cartier for a positive integer $a$. 
Let $H$ be an ample Cartier divisor 
on $X$. Assume that $K_X+B$ is not nef. 
We put 
$$
r=\max \{ t \in \mathbb R \,|\, H+t(K_X+B)  \text{ is nef}\,  \}. 
$$ 
Then $r$ is a rational number of the 
form $u/v$ $($$u, v\in \mathbb Z$$)$ where 
$0<v\leq a(\dim X+1)$. 
\end{thm}

As a corollary, we obtain the cone theorem for 
log canonical pairs. 
It is a formal consequence of the rationality and 
the base point free theorems. 
It is a special case of \cite[Theorem 5.10]{ambro}. 

\begin{thm}[Cone theorem]\label{cone} 
Let $(X, B)$ be a projective 
log canonical pair. 
Then we have 
\begin{itemize}
\item[(i)] There are $($countably many$)$ 
rational curves $C_j\subset X$ such that $0<-(K_X+B)\cdot C_j\leq 
2\dim X$, and 
$$
\overline {NE}(X)=\overline {NE}(X)_{(K_X+B)\geq 0}
+\sum \mathbb R_{\geq 0}[C_j]. 
$$
\item[(ii)] For any $\varepsilon >0$ and ample 
$\mathbb Q$-divisor $H$, 
$$
\overline {NE}(X)=\overline {NE}(X)_{(K_X+B+\varepsilon 
H)\geq 0}
+\underset{\text{finite}}\sum \mathbb R_{\geq 0}[C_j]. 
$$
\item[(iii)] Let $F\subset \overline {NE}(X)$ be a 
$(K_X+B)$-negative 
extremal face. Then there is a unique morphism 
$\varphi_F:X\to Z$ such that 
$(\varphi_F)_*\mathcal O_X\simeq \mathcal O_Z$, 
$Z$ is projective, and 
an irreducible curve $C\subset X$ is mapped to a point 
by $\varphi_F$ if and only if $[C]\in F$. 
The map 
$\varphi_F$ is called the {\em{contraction}} of $F$. 
\item[(iv)] Let $F$ and $\varphi_F$ be as in {\em{(iii)}}. Let $L$ 
be a line bundle on $X$ such that 
$L\cdot C=0$ for every curve $C$ with 
$[C]\in F$. Then there is a line bundle $L_Z$ on $Z$ such that 
$L\simeq \varphi_F^*L_Z$. 
\end{itemize}
\end{thm}

In \cite{ambro}, Ambro did not discuss any generalization 
of Shokurov's non-vanishing theorem. 
Instead, he introduced the notion of {\em{quasi-log varieties}} 
and proved the base point free theorem for quasi-log varieties by 
induction on the dimension. 
His approach is natural but demands some very complicated 
and powerful vanishing and torsion-free theorems for reducible 
varieties. For the details of the theory of quasi-log varieties, 
see \cite{book}. We note that \cite{fujino4} is a quick introduction to 
the theory of quasi-log varieties. 

All the results in this paper are stated and investigated 
in \cite{book} following \cite{ambro}. 
So, the contribution of this paper is to give 
a correct formulation of Shokurov's non-vanishing theorem 
for log canonical pairs and prove it in a simple manner. 
Once we obtain correct formulations of vanishing and 
non-vanishing theorems for log canonical 
pairs (see Theorems \ref{thm5}, \ref{key-vani}, and Theorem \ref{31}), 
there are no difficulties to obtain the fundamental 
theorems for log canonical pairs. 
I hope that this paper will supply a new method to study linear 
systems on log canonical pairs. 
I recommend the reader to see \cite{fujino5}, 
\cite{fujino6}, 
\cite{fujino7}, \cite{fujino8}, and \cite{fujino-takagi} for further 
studies. 

We summarize the contents of this paper. 
In Section \ref{sec2}, we collect some preliminary results 
on vanishing and torsion-free theorems. 
We prove the basic properties of lc centers. 
This section contains no new results. 
In Section \ref{sec3}, we give a proof of the non-vanishing theorem. 
This section is the main part of this paper. 
Our proof is short and very easy to understand. 
Shokurov's concentration method is the main ingredient of Section \ref{sec3}. 
Section \ref{sec4} is devoted to the proof of the 
base point free theorem. The reader will be surprised since  
our proof is very easy and understand that our non-vanishing theorem 
is powerful. 
In Section \ref{sec5}, 
we prove the rationality theorem for 
log canonical pairs. The proof is essentially the same 
as the one for klt pairs. We need only the vanishing theorem 
given in Section \ref{sec2} (cf.~Theorem \ref{thm5}) 
to obtain the rationality theorem. 
In the final section:~Section \ref{sec6}, 
we give a proof of the cone theorem. 
The reader who understands \cite{fujino} can read this paper without 
any difficulties. 

We close this introduction with the following notation. 

\begin{notation}
Let $X$ be a normal variety and $B$ an effective 
$\mathbb Q$-divisor such that 
$K_X+B$ is $\mathbb Q$-Cartier. 
Then we can define the {\em{discrepancy}} 
$a(E, X, B)\in \mathbb Q$ for 
every prime divisor $E$ {\em{over}} $X$. 
If $a(E, X, B)\geq -1$ (resp.~$>-1$) for 
every $E$, then $(X, B)$ is called {\em{log canonical}} 
(resp.~{\em{kawamata log terminal}}). 
We sometimes abbreviate log canonical (resp.~kawamata log terminal) 
to {\em{lc}} (resp.~{\em{klt}}). 

Assume that $(X, B)$ is log canonical. 
If $E$ is a prime divisor over $X$ 
such that 
$a(E, X, B)=-1$, 
then $c_X(E)$ is called a {\em{log canonical center}} 
({\em{lc center}}, for short) 
of $(X, B)$, where $c_X(E)$ is 
the closure of the 
image 
of $E$ on $X$. 

Let $(X, B)$ be a log canonical pair and 
$M$ an effective $\mathbb Q$-divisor on $X$. 
The {\em{log canonical threshold}} of $(X, B)$ with respect to 
$M$ is defined by 
$$
c=\sup \{ t \in \mathbb R \,|\, (X, B+tM) \ 
{\text{is log canonical}}\}. 
$$ 
We can easily check that $c$ is a rational number and 
that $(X, B+cM)$ is lc but not klt. 

Let $(X, B)$ be a log canonical pair. 
Then a {\em{stratum}} of $(X, B)$ denotes $X$ itself or 
an lc center of $(X, B)$. 

Let $Y$ be a smooth variety and $T$ a simple normal crossing divisor 
on $Y$. Then a {\em{stratum}} of $T$ means an lc center of 
the pair $(Y, T)$. 

Let $r$ be a rational number. 
The integral part $\llcorner r\lrcorner$ is 
the largest integer $\leq r$ and 
the fractional part $\{r\}$ is defined by $r-\llcorner r\lrcorner$. 
We put $\ulcorner r\urcorner =-\llcorner -r \lrcorner$ and 
call it the round-up of $r$. 
For a $\mathbb Q$-divisor $D=\sum _{i=1}^r d_i D_i$, 
where $D_i$ is a prime divisor for every $i$ and $D_i\ne D_j$ for 
$i\ne j$, we call $D$ a {\em{boundary}} $\mathbb Q$-divisor 
if $0\leq d_i \leq 1$ for every $i$. 
We note that $\sim_{\mathbb Q}$ denotes 
the $\mathbb Q$-linear equivalence of $\mathbb Q$-Cartier 
$\mathbb Q$-divisors. 
We put $\llcorner D\lrcorner =
\sum \llcorner d_i\lrcorner D_i$, 
$\ulcorner D\urcorner =
\sum \ulcorner d_i\urcorner D_i$, 
$\{D\} =
\sum \{d_i\} D_i$, $D^{<1}=\sum _{d_i<1}d_iD_i$, and 
$D^{=1}=\sum _{d_i=1}D_i$. 

We write $\Bs|L|$ to denote the {\em{base locus}} 
of the linear system $|L|$. 
\end{notation}

\begin{ack}
The author 
was partially supported by The Inamori Foundation and by 
the Grant-in-Aid for Young Scientists (A) $\sharp$20684001 from 
JSPS. He would like to thank Takeshi Abe for discussions. 
\end{ack}

\section{On vanishing and torsion-free theorems}\label{sec2}

In this section, we collect some preliminary results for the 
reader's convenience. 
The next theorem is 
a very special case of \cite[Theorem 3.2]{ambro}. 

\begin{thm}[Torsion-freeness and 
vanishing theorem]\label{thm43}Let 
$Y$ be a smooth projective variety and 
$B$ a boundary $\mathbb Q$-divisor such that 
$\Supp B$ is simple normal crossing. 
Let $f:Y\to X$ be a projective morphism and $L$ a Cartier 
divisor on $Y$ such that 
$H\sim _{\mathbb Q}L-(K_Y+B)$ is $f$-semi-ample. 
\begin{itemize}
\item[(i)] Every non-zero local section of $R^qf_*\mathcal O_Y(L)$ contains 
in its support the $f$-image of 
some stratum of $(Y, B)$. 
\item[(ii)] 
Assume that $H\sim _{\mathbb Q}f^*H'$ for 
some ample $\mathbb Q$-Cartier 
$\mathbb Q$-divisor $H'$ on $X$. 
Then $H^p(X, R^qf_*\mathcal O_Y(L))=0$ for every $p>0$ 
and $q\geq 0$. 
\end{itemize}
\end{thm}

The proof of Theorem \ref{thm43} is not difficult. 
For a short and almost self-contained 
proof, see \cite{fujino}. 
See also \cite[Chapter 2]{book} for a thorough treatment. 
As an application of Theorem \ref{thm43}, 
we prepare the following powerful vanishing theorem. 
It will play basic roles for the study of 
log canonical pairs. 

\begin{thm}[{cf.~\cite[Theorem 4.4]{ambro}}]\label{thm5} 
Let $X$ be a normal projective 
variety and $B$ a boundary $\mathbb Q$-divisor 
on $X$ such that 
$(X, B)$ is log canonical. 
Let $D$ be a Cartier divisor on $X$. 
Assume that $D-(K_X+B)$ is ample. 
Let $\{C_i\}$ be {\em{any}} set of 
lc centers of the pair $(X, B)$. 
We put $W=\bigcup C_i$ with the reduced scheme structure. 
Then we have 
$$
H^i(X, \mathcal I_W\otimes \mathcal O_X(D))=0,  
\ \ H^i(X, \mathcal O_X(D))=0,  
$$
and 
$$
H^i(W, \mathcal O_W(D))=0   
$$
for every $i>0$, where $\mathcal I_W$ is the defining 
ideal sheaf of $W$ on $X$. 
In particular, the restriction 
map 
$$
H^0(X, \mathcal O_X(D))\to H^0(W, \mathcal O_W(D))
$$ 
is surjective. 
\end{thm}

\begin{proof}
Let $f:Y\to X$ be a resolution 
such that $\Supp f^{-1}_*B\cup \Exc (f)$, 
where $\Exc (f)$ is the exceptional locus of 
$f$, is a simple 
normal crossing 
divisor. 
We can further assume that 
$f^{-1}(W)$ is a simple normal crossing 
divisor on $Y$. We can write 
$$
K_Y+B_Y=f^*(K_X+B). 
$$ 
Let $T$ be the union of the irreducible 
components of $B^{=1}_Y$ that 
are mapped into $W$ by $f$. 
We consider 
the following 
short exact sequence 
$$
0\to \mathcal O_Y(A-T)\to \mathcal O_Y(A)\to \mathcal 
O_T(A)\to 0, 
$$ 
where $A=\ulcorner -(B^{<1}_Y)\urcorner$. 
Note that $A$ is an effective 
$f$-exceptional 
divisor. 
We obtain the following long exact sequence 
\begin{align*}
0&\to f_*\mathcal O_Y(A-T)\to f_*\mathcal O_Y(A)
\to f_*\mathcal O_T(A)\\ 
&\overset{\delta}\to R^1f_*\mathcal O_Y(A-T)\to \cdots.  
\end{align*}
Since 
$$
A-T-(K_Y+\{B_Y\}+B^{=1}_Y-T)=-(K_Y+B_Y)\sim_{\mathbb Q}
-f^*(K_X+B), 
$$ 
every non-zero local section of $R^1f_*\mathcal O_Y(A-T)$ contains 
in its support the 
$f$-image of some stratum of $(Y, \{B_Y\}+B^{=1}_Y-T)$ by 
Theorem \ref{thm43} (i). 
On the other hand, $W=f(T)$. Therefore, 
the connecting homomorphism $\delta$ is the zero map. 
Thus, we have a short 
exact sequence 
$$
0\to f_*\mathcal O_Y(A-T)\to \mathcal O_X\to 
f_*\mathcal O_T(A)\to 0. 
$$
So, we obtain $f_*\mathcal O_T(A)\simeq \mathcal O_W$ and 
$f_*\mathcal O_Y(A-T)\simeq \mathcal I_W$, the 
defining ideal sheaf of $W$. 
The isomorphism $f_*\mathcal O_T(A)\simeq \mathcal O_W$ 
plays crucial roles in this paper. 
Thus we write it as a lemma. 
\begin{lem}\label{23}
We have $f_*\mathcal O_T(A)\simeq \mathcal O_W$. 
It obviously implies that $f_*\mathcal O_T\simeq \mathcal O_W$. 
\end{lem}
Since 
$$
f^*D+A-T-(K_Y+\{B_Y\}+B^{=1}_Y-T)\sim_{\mathbb Q}
f^*(D-(K_X+B)), 
$$ 
and 
$$
f^*D+A-(K_Y+\{B_Y\}+B^{=1}_Y)\sim_{\mathbb Q}
f^*(D-(K_X+B)), 
$$
we have 
$$
H^i(X, \mathcal I_W\otimes \mathcal O_X(D))\simeq 
H^i(X, f_*\mathcal O_Y(A-T)\otimes \mathcal O_X(D))=0 
$$ 
and 
$$
H^i(X, \mathcal O_X(D))\simeq 
H^i(X, f_*\mathcal O_Y(A)\otimes \mathcal O_X(D))=0 
$$ 
for every $i>0$ by Theorem \ref{thm43} (ii). 
By the long exact sequence 
\begin{align*}
\cdots \to H^i(X, \mathcal O_X(D))&\to 
H^i(W, \mathcal O_W(D))\\ &\to H^{i+1}(X, \mathcal I_W\otimes 
\mathcal O_X(D))\to \cdots, 
\end{align*}
we have $H^i(W, \mathcal O_W(D))=0$ for 
every $i>0$. 
We finish the proof. 
\end{proof}

As a corollary, we can easily check the following 
result (cf.~\cite[Propositions 4.7 and 4.8]{ambro}). 

\begin{thm}\label{2424}
Let $X$ be a normal projective variety and $B$ an effective 
$\mathbb Q$-divisor 
such that 
$(X, B)$ is log canonical. 
Then we have the following properties. 
\begin{itemize}
\item[(1)] $(X, B)$ has at most finitely many lc centers. 
\item[(2)] An intersection of two lc centers is 
a union of lc centers. 
\item[(3)] Any union of lc centers of $(X, B)$ is semi-normal. 
\item[(4)] Let $x\in X$ be a closed point such that 
$(X, B)$ is lc but not klt at $x$. Then 
there is a unique minimal lc center $W_x$ passing through 
$x$, and $W_x$ is normal at $x$. 
\end{itemize}
\end{thm}
\begin{proof} 
We use the notation in the proof of Theorem \ref{thm5}. 
(1) is obvious. 
(3) is also obvious by Lemma \ref{23} since $T$ is a simple normal crossing 
divisor. 
Let $C_1$ and $C_2$ be two lc centers of $(X, B)$. 
We fix a closed point $P\in C_1\cap C_2$. 
For the 
proof of (2), it is enough to find an lc center $C$ such that $P\in C\subset C_1\cap C_2$. 
We put 
$W=C_1\cup C_2$. 
By Lemma \ref{23}, 
we obtain $f_*\mathcal O_T\simeq \mathcal O_W$. This means 
that $f:T\to W$ has connected fibers. 
We note that $T$ is a simple normal crossing divisor on $Y$. 
Thus, there exist irreducible 
components $T_1$ and $T_2$ of $T$ such that  
$T_1\cap T_2\cap f^{-1}(P)\ne \emptyset$ and that 
$f(T_i)\subset C_i$ for $i=1, 2$. 
Therefore, we can find an lc center $C$ with 
$P\in C\subset C_1\cap C_2$. 
We finish the proof of (2). 
Finally, we will prove (4). 
The existence and the uniqueness of 
the minimal lc center follow from (2). 
We take the unique minimal lc center $W=W_x$ passing through 
$x$. By Lemma \ref{23}, 
we have $f_*\mathcal O_T\simeq \mathcal O_W$. 
By shrinking $W$ around $x$, we can assume that 
every stratum of $T$ dominates $W$. 
Thus, $f:T\to W$ factors through 
the normalization $W^{\nu}$ of $W$. Since 
$f_*\mathcal O_T\simeq \mathcal O_W$, we obtain 
that $W^{\nu}\to W$ is an isomorphism. So, we 
obtain (4). 
\end{proof}

We close this section with the following 
very useful vanishing theorem. 
It is a special case of \cite[Theorem 4.4]{ambro}. 
For details, see \cite[Theorem 3.39]{book}. 
Here, we give a quick reduction to Theorem \ref{thm43} (ii) 
by using \cite{bchm} for the reader's convenience.   

\begin{thm}\label{key-vani}
Let $(X, B)$ be a projective lc pair and 
$W$ a minimal lc center of $(X, B)$. 
Let $D$ be a Cartier divisor on $W$ such that 
$D-(K_X+B)|_W$ is ample. 
Then $H^i(W, \mathcal O_W(D))=0$ for every $i>0$. 
\end{thm}

\begin{proof}
By Hacon (cf.~\cite[Theorem 3.1]{ko-ko}), 
we can make a projective 
birational morphism $f:Y\to X$ such that 
$K_Y+B_Y=f^*(K_X+B)$ and that $(Y, B_Y)$ is dlt. 
It is an application of the results in \cite{bchm}. 
For the definition and the basic properties 
of {\em{dlt}} pairs, see \cite[Section 2.3]{km} 
and \cite{what}. 
We take an lc center $V$ of $(Y, B_Y)$ such that 
$f(V)=W$ and put $K_V+B_V=(K_Y+B_Y)|_V$. 
Then $(V, B_V)$ is dlt (cf.~\cite[Section 3.9]{what}) and 
$K_V+B_V\sim _{\mathbb Q}f^*((K_X+B)|_W)$. 
Let $g:Z\to V$ be a resolution such that 
$K_Z+B_Z=g^*(K_V+B_V)$ and that $\Supp B_Z$ is simple normal 
crossing. Then we have 
$K_Z+B_Z\sim _{\mathbb Q}h^*((K_X+B)|_W)$, where 
$h=f\circ g$. 
Since 
$$h^*(D-(K_X+B)|_W)\sim _{\mathbb Q}h^*D+\ulcorner 
-(B^{<1}_Z)\urcorner-(K_Z+B^{=1}_Z+\{B_Z\}), $$
we obtain 
$$H^i(W, h_*\mathcal O_Z(h^*D+\ulcorner -(B^{<1}_Z)\urcorner))=0$$ for 
every $i>0$ by Theorem \ref{thm43} (ii). 
We note that 
$$h_*\mathcal O_Z(h^*D+\ulcorner -(B^{<1}_Z)\urcorner)\simeq 
f_*\mathcal O_V(f^*D)$$ by the projection formula since 
$\ulcorner -(B^{<1}_Z)\urcorner$ is effective and $g$-exceptional. 
We note that 
$\mathcal O_W(D)$ is a direct summand of $f_*\mathcal O_V(f^*D)\simeq 
\mathcal O_W(D)\otimes f_*\mathcal O_V$ since $W$ is normal 
(cf.~Theorem \ref{2424} (4)). 
Therefore, we have $H^i(W, \mathcal O_W(D))=0$ for every $i>0$. 
\end{proof}

\begin{rem}
We can prove Theorem \ref{key-vani} without 
using \cite{bchm}. 
For the original argument, see \cite[Theorem 4.4]{ambro} 
and \cite[Theorem 3.39]{book}. It depends on 
the theory of mixed Hodge structures (cf.~\cite[Chapter 2]{book}). 
\end{rem}

\section{Non-vanishing theorem}\label{sec3}

In this section, we prove the non-vanishing theorem, 
which is the main theorem of this paper. 
The proof given here is very easy. 

\begin{proof}[{Proof of {\em{Theorem \ref{31}}}}]
Let $W$ be a minimal 
lc center of $(X, B)$. 
If $L|_W$ is numerically trivial, 
then we have 
$$
h^0(W, \mathcal O_W(L))=\chi 
(W, \mathcal O_W(L))=\chi (W, \mathcal O_W)=h^0(W, \mathcal O_W)=1$$ 
by 
\cite[Chapter II \S2 Theorem 1]{kleiman} 
and the vanishing theorem (see Theorem \ref{key-vani}). 
Therefore, 
$L|_W$ is linearly trivial since 
$L|_W$ is numerically trivial. In particular, $|mL|_W|$ is free for 
every $m>0$. 
On the other hand, 
$$
H^0(X, \mathcal O_X(mL))\to H^0(W, \mathcal O_W(mL))
$$
is surjective for every $m\geq a$ by Theorem \ref{thm5}. 
Thus, $\Bs |mL|$ does not contain $W$ for every $m\geq a$.  

Assume that 
$L|_W$ is not numerically trivial. Let $x\in W$ be a general smooth point. 
If $l$ is a sufficiently large integer, 
then we can find an effective Cartier divisor 
$N$ on $W$ such that $N\sim b(lL-(K_X+B))|_W$ with 
$\mult _x N>b\dim W$ for some 
positive divisible integer $b$ by Shokurov's concentration method. 
See, for example, \cite[3.5 Step 2]{km}. 
If $b$ is sufficiently large and divisible, 
then $\mathcal I_W\otimes \mathcal O_X(b(lL-(K_X+B)))$ 
is generated by global sections and 
$H^1(X, \mathcal I_W\otimes \mathcal O_X(b(lL-(K_X+B))))=0$ 
since $lL-(K_X+B)$ is ample, where $\mathcal I_W$ is the defining 
ideal sheaf of $W$ on $X$. 
By using the following short exact sequence 
\begin{align*}
0&\to H^0(X, \mathcal I_W\otimes \mathcal O_X(b(lL-(K_X+B))))\\
&\to H^0(X, \mathcal O_X(b(lL-(K_X+B))))\\ &\to 
H^0(W, \mathcal O_W(b(lL-(K_X+B))))\to 0, 
\end{align*} 
we can find an effective $\mathbb Q$-divisor $M$ on $X$ with 
the following properties. 
\begin{itemize}
\item[(i)] $M|_W$ is an effective $\mathbb Q$-divisor 
such that $\mult _x M|_W>\dim W$. 
\item[(ii)] $M\sim _{\mathbb Q}lL-(K_X+B)$ for some 
positive large integer $l$. 
\item[(iii)] $(X, B+M)$ is lc outside $W$. 
\end{itemize} 
We take the log canonical threshold $c$ of $(X, B)$ with 
respect to $M$. 
Then $(X, B+cM)$ is lc but not klt. 
By the above construction, we have $0<c<1$.  
We replace $(X, B)$ with $(X, B+cM)$, 
$a$ with $a-ac+cl$. 
Then we have that 
$$
(a-ac+cl)L-(K_X+B+cM)\sim _{\mathbb Q}(1-c)(aL-(K_X+B)) 
$$ 
is ample. Moreover, 
we can find a smaller lc center $W'$ of $(X, B+cM)$ contained 
in $W$. By repeating 
this process, we reach the situation where 
$L|_W$ is numerically trivial. 

Anyway, we proved that $\Bs|mL|$ contains no lc centers 
of $(X, B)$ for every $m\gg 0$. 
\end{proof}

\section{Base point free theorem}\label{sec4}

We give a proof of Theorem \ref{main1}. 
Our proof is much easier than Ambro's proof 
for quasi-log varieties. 

\begin{proof}[{Proof of {\em{Theorem \ref{main1}}}}]
If $L$ is numerically trivial, 
then 
$$h^0(X, \mathcal O_X(\pm L))=\chi (X, \mathcal O_X(\pm L))
=\chi (X, \mathcal O_X)=h^0(X, \mathcal O_X)=1$$ by 
\cite[Chapter II \S2 Theorem 1]{kleiman} and 
the vanishing theorem (cf.~Theorem \ref{thm5}). 
Thus, $L$ is linearly trivial. 
In this case, $|mL|$ is free for every $m$. 
So, from now on, we can assume that 
$L$ is not numerically trivial. 

We assume that 
$(X, B)$ is klt. 
Let $x\in X$ be a general smooth point. 
Then we can find an effective 
$\mathbb Q$-divisor 
$M$ on $X$ such that 
$$
M\sim _{\mathbb Q}lL-(K_X+B)
$$ 
for some large integer $l$ and 
that $\mult _x M>n=\dim X$. 
It is well known as Shokurov's concentration method. 
See, for example, \cite[3.5 Step 2]{km}.   
Let $c$ be the log canonical threshold 
of $(X, B)$ with respect to $M$. 
By construction, 
we have $0<c<1$. 
Then 
$$
(a-ac+cl)L-(K_X+B+cM)\sim _{\mathbb Q}(1-c)(aL-(K_X+B)) 
$$ 
is ample. 
Therefore, by replacing $B$ with 
$B+cM$, 
$a$ with $a-ac+cl$, we 
can assume that $(X, B)$ is lc but not klt. 

From now on, we assume that $(X, B)$ is lc but not klt and 
that $L$ is not numerically trivial.  
By Theorem \ref{31}, 
we can take general members $D_1, \cdots, D_{n+1}\in |p^{m_1}L|$ for some 
prime integer $p$ and a positive integer $m_1$. 
Since $D_1, \cdots, D_{n+1}$ are general, 
$(X, B+D_1+\cdots +D_{n+1})$ is lc outside $\Bs|p^{m_1}L|$. 
It is easy to see that 
$(X, B+D)$, where $D=D_1+\cdots +D_{n+1}$, is not lc 
at the generic point of any irreducible 
component of $\Bs|p^{m_1}L|$. Let $c$ 
be the log canonical threshold 
of $(X, B)$ with 
respect to $D$. 
Then $(X, B+cD)$ is lc but not klt, and $0<c<1$. 
We note that 
$$
(c(n+1)p^{m_1}+a)L-(K_X+B+cD)\sim _{\mathbb Q}aL-(K_X+B)
$$ 
is ample. 
By construction, there 
exists an lc center of $(X, B+cD)$ contained 
in $\Bs|p^{m_1}L|$. By Theorem \ref{31}, 
we can find $m_2>m_1$ such 
that 
$\Bs|p^{m_2}L|\subsetneq \Bs|p^{m_1}L|$. By 
noetherian induction, there exists $m_k$ such that 
$\Bs|p^{m_k}L|=\emptyset$. 
Let $p'$ be a prime integer such that 
$p'\ne p$. 
Then, by the same argument, we can prove $\Bs|p'^{m'_{k'}}L|=\emptyset$ for 
some positive integer $m'_{k'}$. 
So, there exists a positive number $m_0$ such that 
$|mL|$ is free for every $m\geq m_0$. 
\end{proof}

\section{Rationality theorem}\label{sec5}

Here, we prove the rationality theorem 
for log canonical pairs. 
Before we  start the proof, we recall the following lemmas. 
\begin{lem}[{cf.~\cite[Lemma 3.19]{km}}]\label{62}  
Let $P(x, y)$ be a non-trivial 
polynomial of degree $\leq n$ and 
assume that 
$P$ vanishes for all sufficiently large integral solutions of 
$0<ay-rx<\varepsilon$ for some fixed positive integer $a$ and 
positive $\varepsilon$ for some $r\in \mathbb R$. 
Then $r$ is rational, and 
in reduced form, $r$ has denominator $\leq 
a(n+1)/\varepsilon$. 
\end{lem}

For the proof, see \cite[Lemma 3.19]{km}. 

\begin{lem}\label{63}  
Let $C$ be a projective variety and $D_1$, $D_2$ 
Cartier divisors on $X$. 
Consider the Hilbert polynomial 
$$
P(u_1, u_2)=\chi (C, \mathcal O_C(u_1D_1+u_2D_2)). 
$$ 
It is a polynomial in $u_1$ and $u_2$ of total 
degree $\leq \dim C$ {\em{(}}cf.~\cite[Theorem (Snapper)]{kleiman}{\em{)}}. 
If $D_1$ is ample, then 
$P(u_1, u_2)$ is 
a non-trivial polynomial. 
It is because $P(u_1, 0)=h^0(C, \mathcal O_C(u_1D_1))\ne 0$ 
if $u_1$ is sufficiently large. 
\end{lem}

\begin{proof}[Proof of {\em{Theorem \ref{61}}}]
By using $mH$ with various large $m$ in place of $H$, we 
can assume that $H$ is very ample (cf.~\cite[3.4 Step 1]{km}). 
We put $\omega=K_X+B$ for simplicity. 
For each $(p, q)\in \mathbb Z^2$, let $L(p, q)$ denote the base 
locus of the linear system $|M(p, q)|$ on $X$ (with 
reduced scheme structure), 
where $M(p, q)=pH+qa\omega$. By 
definition, $L(p, q)=X$ if and only if $|M(p, q)|=\emptyset$. 
\setcounter{cla}{0}
\begin{cla}[{cf.~\cite[Claim 3.20]{km}}]\label{claim1}  
Let $\varepsilon$ be a positive number. 
For $(p, q)$ sufficiently 
large and $0<aq-rp<\varepsilon$, $L(p, q)$ is the same 
subset of $X$. 
We call this subset $L_0$. 
We let $I\subset \mathbb Z\times \mathbb Z$ be the set of $(p, q)$ for 
which $0<aq-rp<1$ and $L(p, q)=L_0$. 
We note that $I$ contains all sufficiently 
large $(p, q)$ with $0<aq-rp<1$. 
\end{cla}

For the proof, see \cite[Claim 3.20]{km}. 

\begin{cla}\label{claim2} 
We assume that $r$ is not rational or that 
$r$ is rational and has denominator $>a(n+1)$ in reduced form, 
where $n=\dim X$. 
Then, for $(p, q)$ sufficiently large and $0<aq-rp<1$, 
$\mathcal O_X(M(p, q))$ is generated by global sections 
at the generic point of any lc center 
of $(X, B)$. 
\end{cla}

\begin{proof}[Proof of {\em{Claim \ref{claim2}}}]
We note that $M(p, q)-\omega=pH+(qa-1)\omega$. If 
$aq-rp<1$ and $(p, q)$ is sufficiently large, then $M(p, q)-\omega$ is 
ample. Let $C$ be an lc center of $(X, B)$. 
Then $P_C(p, q)=\chi (C, \mathcal O_C(M(p, q)))$ is a non-trivial 
polynomial of degree 
at most $\dim C\leq \dim X$ by Lemma \ref{63}. 
By Lemma \ref{62}, there exists $(p, q)$ such that 
$P_C(p, q)\ne 0$ and 
that $(p, q)$ sufficiently large and $0<aq-rp<1$. 
By the ampleness of $M(p, q)-\omega$, $P_C(p, q)=
\chi (C, \mathcal O_C(M(p, q)))
=h^0(C, \mathcal O_C(M(p, q)))$ and 
$
H^0(X, \mathcal O_X(M(p, q)))\to H^0(C, \mathcal O_C(M(p, q)))
$ 
is surjective by Theorem \ref{thm5}. 
Therefore, $\mathcal O_X(M(p, q))$ 
is generated 
by global sections at the generic point of $C$. 
By combining this with Claim \ref{claim1}, $\mathcal O_X(M(p, q))$ is 
generated by 
global sections at the generic point of any 
lc center of $(X, B)$ if 
$(p, q)$ is sufficiently large with 
$0<aq-rp<1$. 
So, we obtain Claim \ref{claim2}. 
\end{proof}

Note 
that $\mathcal O_X(M(p, q))$ is not generated by 
global sections because $M(p, q)$ is not nef. 
Therefore, $L_0\ne \emptyset$. 
Let $D_1, \cdots, D_{n+1}$ 
be general members of $|M(p_0, q_0)|$ with 
$(p_0, q_0)\in I$. 
Then $K_X+B+\sum _{i=1}^{n+1}D_i$ 
is not lc at the generic point of any irreducible 
component of $L_0$ and is lc outside $L_0$. 
Let $c$ be the log canonical threshold of $(X, B)$ with 
respect to $D=\sum _{i=1}^{n+1}D_i$. 
Then, we have $0<c<1$ and 
that $K_X+B+cD$ is lc but not klt. 
Note that $c>0$ by Claim \ref{claim2}. 
Thus, the lc pair $(X, B+cD)$ has some 
lc centers contained in $L_0$. Let $C$ be an lc center 
contained in $L_0$. 
We consider $K_X+B+c D=\omega+c D 
\sim _{\mathbb Q}c(n+1)p_0H+(1+c(n+1)q_0a)\omega$. 
We put $\omega'=K_X+B+c D$ for 
simplicity. Thus we have 
$pH+qa\omega-\omega'\sim _{\mathbb Q}
(p-c(n+1)p_0)H+(qa-(1+c(n+1)q_0a))\omega$. 
If $p$ and $q$ are large enough and $0<aq-rp\leq aq_0-rp_0$, 
then $pH+qa\omega-\omega'$ is ample. 
It is because 
\begin{align*}
&(p-c(n+1)p_0)H+(qa-(1+c(n+1)q_0a))\omega\\
&=(p-(1+c(n+1))p_0)H+(qa-(1+c(n+1))q_0a)\omega+p_0H+(q_0a-1)\omega.  
\end{align*}

Suppose that $r$ is not rational. 
There must be arbitrarily large $(p, q)$ such that 
$0<aq-rp<\varepsilon =aq_0-rp_0$ and 
$\chi(C, \mathcal O_C(M(p, q)))\ne 0$ by Lemma \ref{62} 
because $P_C(p ,q)=\chi (C, \mathcal O_C(M(p, q)))$ 
is a non-trivial polynomial of degree at most $\dim C$ by 
Lemma \ref{63}. 
Since $M(p, q)-\omega'$ is ample by 
$0<aq-rp<aq_0-rp_0$, 
we have $h^0(C, \mathcal O_C(M(p, q)))=
\chi(C, \mathcal O_C(M(p, q)))\ne 0$ 
by the vanishing theorem (cf.~Theorem \ref{thm5}). 
By the vanishing theorem:~Theorem \ref{thm5}, 
the restriction morphism $$
H^0(X, \mathcal O_X(M(p, q)))\to H^0(C, \mathcal O_C(M(p, q)))
$$ 
is surjective because $M(p, q)-\omega'$ is ample. 
We note that $C$ is an lc center of $(X, B+cD)$. 
Thus $C$ is not contained in $L(p, q)$. Therefore, 
$L(p, q)$ is a proper subset of $L(p_0, q_0)=L_0$, 
giving the desired contradiction. 
So now we know that $r$ is rational. 

We next suppose that the assertion of the theorem concerning 
the denominator of $r$ is false. 
Choose $(p_0, q_0)\in I$ such that 
$aq_0-rp_0$ achieves its maximum value, 
which we can assume has the form $d/v$. 
If $0<aq-rp\leq d/v$ and 
$(p, q)$ is sufficiently large, then $\chi (C, 
\mathcal O_C(M(p, q)))=h^0(C, \mathcal O_C(M(p, q)))$ since 
$M(p, q)-\omega'$ is ample. 
There exists sufficiently 
large $(p, q)$ in the strip $0<aq-rp<1$ with $\varepsilon =1$ 
for which 
$h^0(C, \mathcal O_C(M(p, q)))=\chi (C, \mathcal O_C(M(p, q)))\ne 0$ by 
Lemma \ref{62} since 
$\chi (C, \mathcal O_C(M(p, q)))$ is a non-trivial 
polynomial of degree at most $\dim C$ by Lemma \ref{63}. 
Note that $aq-rp\leq d/v=aq_0-rp_0$ holds automatically for 
$(p, q)\in I$. 
Since 
$H^0(X, \mathcal O_X(M(p, q)))\to H^0(C, \mathcal O_C(M(p,q)))$ 
is surjective by the ampleness of $M(p,q)-\omega'$, we 
obtain the desired contradiction by the same reason as above. 
So, we finish the proof. 
\end{proof}

\section{Cone theorem}\label{sec6}

In this final section, we give a proof of the cone 
theorem for log canonical pairs. 

\begin{proof}[Proof of {\em{Theorem \ref{cone}}}] 
The estimate $\leq 2\dim X$ in (i) can be proved 
by Kawamata's argument in \cite{kawamata} with the aid of \cite{bchm}. 
For details, see \cite[Subsection 3.1.3]{book} 
or \cite[Section 18]{fujino8}. 
The other statements in (i) and (ii) are formal consequences of 
the rationality theorem. For the proof, 
see \cite[Theorem 3.15]{km}. 
The statements (iii) and (iv) are obvious by Theorem \ref{main1} and 
the statements (i) and (ii). See Steps 7 and 9 in \cite[3.3 The Cone 
Theorem]{km}. 
\end{proof}

\ifx\undefined\bysame
\newcommand{\bysame|{leavemode\hbox to3em{\hrulefill}\,}
\fi

\end{document}